\documentclass[11pt]{amsart}
\usepackage{amsmath,amscd,amssymb,amsfonts,amsthm,hyperref}

\usepackage[cmtip,matrix,arrow]{xy}

\newtheorem{theorem}{Theorem}[section]

\newtheorem{proposition}[theorem]{Proposition}

\newtheorem{lemma}[theorem]{Lemma}
\theoremstyle{definition}    
\newtheorem{definition}[theorem]{Definition}
\theoremstyle{remark}

\newtheorem{remark}[theorem]{Remark}

\newtheorem{example}[theorem]{Example}

\newcommand\A{\mathcal{A}}

\newcommand{\W}{\mathcal{W}}

\renewcommand{\L}{\mathcal{L}}
\renewcommand{\O}{\mathcal{O}}

\newcommand{\J}{\mathcal{J}}

\newcommand{\ca}{\mathcal}

\newcommand{\E}{\ca{E}}

\newcommand{\N}{\mathbb{N}}
\newcommand{\R}{\mathbb{R}}

\newcommand\pt{\on{pt}}

\newcommand\lie[1]{\mathfrak{#1}}

\newcommand{\h}{\lie{h}}

\newcommand{\m}{\lie{m}}

\newcommand{\on}{\operatorname}

\renewcommand{\ker}{ \on{ker}}

\newcommand{\Mult}{  \on{Mult}}

\newcommand{\sz}{\mathsf{s}}
\newcommand{\tz}{\mathsf{t}}

\newcommand\qu{/\kern-.7ex/} 

\newcommand{\lra}{\longrightarrow}
\newcommand{\hra}{\hookrightarrow}
\newcommand{\ra}{\rightarrow}

\renewcommand{\d}{{\mbox{d}}}
\newcommand{\ol}{\overline}

\newcommand{\f}{\frac}

\newcommand{\p}{\partial}

\newcommand\hh{{\f{1}{2}}}

\newcommand{\ti}{\tilde}
\newcommand{\eeq}{\end{eqnarray*}}
\newcommand{\beq}{\begin{eqnarray*}}

\newcommand{\pr}{\on{pr}}

\newcommand{\lin}{{\on{lin}}}
\newcommand{\sA}{\mathsf{A}}
\newcommand{\mf}{\mathfrak}
\newcommand{\rra}{\rightrightarrows}
\renewcommand{\subset}{\subseteq}
\renewcommand{\supset}{\supseteq}
\newcommand{\I}{\ca{I}}

\begin{document}

\title[Euler-like vector fields]{Euler-like vector fields, normal forms,\\ and  isotropic embeddings}

\author{Eckhard Meinrenken}

\begin{abstract}
As shown in \cite{bur:spl}, germs of tubular neighborhood embeddings for submanifolds $N\subset M$ are in one-one correspondence with germs of Euler-like vector fields near $N$. In many contexts, this reduces the proof of `normal forms results' for geometric  structures to the construction of an Euler-like vector field compatible with the given structure. We illustrate this principle in a variety of examples, including the Morse-Bott lemma, Weinstein's Lagrangian embedding theorem, and Zung's linearization theorem for proper Lie groupoids. In the second part of this article, we extend the theory to a weighted context, with an application to isotropic embeddings. 
\end{abstract}

\maketitle
\begin{quote}
	{\it \small To the memory of Hans Duistermaat.}
\end{quote}
\vskip1cm
\bigskip
\section{Introduction}
	 A vector field on a manifold $M$ is called Euler-like with respect to a submanifold $N$ if it vanishes along the submanifold, and its linear approximation is the Euler vector field on the normal bundle. It was observed in \cite{bur:spl} that Euler-like vector fields determine tubular neighborhood embedding of the normal bundle in which they become \emph{the} Euler vector field. This provides a 1-1 correspondence between germs of Euler-like vector fields and germs of tubular neighborhood embeddings. 

Many normal form results for geometric structures along submanifolds $N\subset M$ amount to a construction of tubular neighborhood embeddings $\nu(M,N)\supset O\stackrel{\varphi}{\lra} M$ in which the given structure takes on a simple form -- for example, it becomes linear or quadratic in an appropriate sense. Such normal forms may be reformulated as the existence of Euler-like vector fields that are compatible with the given structure.  In \cite{bis:def,bur:spl}, this perspective was used to prove old and new `splitting theorems'  for Lie algebroids, Poisson structures, generalized complex structures, $L_\infty$-algebroids, and so on.  

One of the purposes of this article is  to explain applications to other types of normal form results, such as the Weinstein Lagrangian embedding theorem, normal forms for Morse-Bott  functions, the Grabowski-Rodkievicz characterization of subbundles of vector bundles, and the linearizability of proper Lie groupoids. This part of the paper is largely expository, 
showing how all of these results can be proved along very similar lines.  

Our second objective is to indicate a generalization to a \emph{weighted} version of the theory. In this paper, we shall only consider the case of degrees $2$ weightings  -- intuitively, this means that some directions normal to  $N$ will be treated 
as `second order'. More precisely, given a subbundle of the normal bundle $F\subset \nu(M,N)$, we define a filtration on functions by assigning weight $1$ to functions vanishing along $N$, and weight $2$ if furthermore the differential 
vanishes in the normal direction given by $F$. It turns out that these data determine a fiber bundle 
\[ \nu_\W(M,N)\to N\]
which we call a \emph{weighted normal bundle}. This fiber bundle is not naturally a vector bundle, but is a `graded bundle' in the sense of Grabowski-Rodkievicz \cite{gra:gra}. There is a notion of  Euler-like vector fields for this context, and given any such vector field one obtains a  tubular neighborhood embedding $\nu_\W(M,N)\subset O\to M$. 

As an application, if $(M,\omega)$ is a symplectic manifold and $N$ is an isotropic submanifold, we consider the weighting defined by the subbundle $F=TN^\omega/TN$. We find that  $\nu_\W(M,N)$  canonically inherits a symplectic 2-form  homogeneous of degree $2$. We will show how to construct a weighted Euler-like vector field $X$ satisfying $\L_X\omega=2\omega$, which then produces a tubular neighborhood embedding preserving symplectic forms. This may be seen as a version of Weinstein's \emph{isotropic embedding theorem} \cite{wei:sym1,we:nei}, but with a local model that 
does not require a  choice of connection. \medskip

{\bf Acknowledgement.} I am grateful to the referee for detailed comments and helpful suggestions. I also thank Francis Bischoff for suggesting the application  to linearizations of groupoids, and Yiannis Loizides for discussions. This research was funded by an NSERC discovery grant.

\bigskip
{\bf Notation.}
We shall use the convention that the flow $(t,m)\mapsto \Phi_t(m)$ of a vector field $X$ is described 
in terms of the Lie derivative $\L_X$ on functions $f\in C^\infty(M)$ by 
\[ \f{d}{d t} f(\Phi_{-t}(m))=(\L_X f)(\Phi_{-t}(m)),\]
with $\Phi_0(m)=m$. The flow is defined for $(t,m)$ in some neighborhood of $\{0\}\times M$ in $\R\times M$.
More generally, given a time dependent vector field $X_t$, we define $\Phi_t$ by the same formula, with $\L_X$ replaced by $\L_{X_t}$. For a vector bundle $V\to M$, the \emph{Euler vector field}  $\E$ is the vector field having scalar multiplication by $e^{-t}$ as its flow. For $V=\R^n$, this is the vector field $\E=\sum_i x_i\f{\p}{\p x_i}$.

\section{Euler-like vector fields on $\R^n$}
\subsection{Linearization of Euler-like vector fields}\label{subsec:lin}
Consider a vector field 
\begin{equation}\label{eq:X} X=\sum_{i=1}^n a_i(x)\f{\p}{\p x_i}\end{equation}
with a critical point at $0$, that is, $a_i(0)=0$. The \emph{linear approximation} of $X$ is the vector field $X_{[0]}$ 
obtained by replacing each coefficient function $a_i$ with the linear term of its Taylor approximation.  (The subscript signifies the degree of homogeneity with respect to scalar multiplication.) The vector field is called (smoothly) \emph{linearizable} 
if there exists a germ of a diffeomorphism $\varphi$ at $0$ such that $\varphi^*X=X_{[0]}$. The \emph{Sternberg linearization theorem} \cite{ste:loc,ste:str} gives a sufficient condition:  $X$ is linearizable provided that the eigenvalues $\lambda_1,\ldots,\lambda_n$  
of the matrix with entries
\[  \f{\p a_i}{\p x_j}(0)\] are \emph{non-resonant}, that is, sums $\lambda_{i_1}+\cdots+\lambda_{i_N}$ with $N>1$ 
(with possible repetitions) are \emph{not}  eigenvalues. 

Sternberg's theorem applies, in particular, when all the eigenvalues are equal to $1$. Here, the linear approximation $X_{[0]}$ is the \emph{Euler vector field}
\[ \E=\sum_{i=1}^n x_i\f{\p}{\p x_i}.\]
In this very special case, a short Moser-type argument 
is available \cite{bur:spl}, with an explicit choice for $\varphi$:

\begin{lemma}\label{lem:2}
If $X\in \mf{X}(\R^n)$ is Euler-like, there exists a unique germ at $0$ of a diffeomorphism $\varphi$ of $\R^n$ 
such that $\varphi(0)=0,\ D\varphi(0)=\on{id}$, and 
\[ \varphi^* X=\E.\]	
\end{lemma}
\begin{proof} 
Since $X$ has linear approximation $X_{[0]}=\E$, the family of  vector fields, for $t\neq 0$,
\[ X_t=\sum_{i=1}^n \f{a_i(t x)}{t}\ \f{\p}{\p x_i}\]
extends smoothly to $t=0$, by $X_0=\E$. The $t$-derivative of this vector field is 
\[ \f{d X_t}{d t}=\f{1}{t}[\E,\ X_t].\]
Define a family of vector fields,  for $t\neq 0$,
\begin{equation} W_t=\f{1}{t} (X_t-\E)=\sum_{i=1}^n \f{a_i(t x)-t x_i }{t^2}\ \f{\p}{\p x_i};\end{equation}
again this extends smoothly to $t=0$. Let $t\mapsto \Phi_t$ be the (germ of) flow of the time dependent vector field $W_t$, with $\Phi_0=\on{id}$. 
The calculation 
\[ \f{d}{d t} \Phi_{-t}^* X_t=\Phi_{-t}^* \Big([W_t,X_t]+\f{d X_t}{d t}\Big)
=\Phi_{-t}^* \big([W_t+\f{1}{t}\E,X_t]\big)=\Phi_{-t}^* \big[\f{1}{t}X_t,X_t\big]=0\]
shows that $ \Phi_{-t}^* X_t$ is constant. Hence, $\varphi=\Phi_{-1}$ satisfies 
\[ \varphi^*X=\Phi_{-1}^*X_1=\Phi_0^*X_0=\E.\]

For any given $t$, the vector field $W_t$ vanishes to 
second order at $x=0$, hence its flow satisfies $D\Phi_t(0)=\on{id}$ for all $t$; hence also $D\varphi(0)=\on{id}$. 
This proves existence; the uniqueness follows because the germ at $0$ of a diffeomorphism of $\R^n$ preserving the Euler vector field is the germ of a 
map commuting with scalar multiplication, hence the germ of a \emph{linear} map. 
\end{proof}

\subsection{Applications}
Here are some applications of the lemma to well-known classical results. 

\begin{example}[Morse Lemma \cite{mor:cal}] \label{ex:morse}
	Let $f\in C^\infty(\R^n)$ be a smooth function with a non-degenerate critical point at $0$, and with $f(0)=0$.  
	Morse's Lemma 
	states that there is a diffeomorphism $\varphi$, defined near $0$, such that $\varphi^*f$ is quadratic.   
	Here is a proof using 	Euler-like vector fields. 
	By Taylor's theorem
	\[ f(x)=\hh \sum_{ij} A_{ij}(x) x_i x_j\]
	for a smooth matrix-valued function $x\mapsto A(x)$ with $A_{ij}(x)=A_{ji}(x)$. 
	The derivatives of $f$ are
	\[ \f{\p f}{\p x_j}=\sum_k B_{jk}(x) x_k\]
	with $B_{jk}=A_{jk}+\hh \sum_r \f{\p A_{jk}}{\p x_r}x_r$. Since $B(0)=A(0)$ is invertible, the matrix-valued function $x\mapsto B(x)$ is invertible for $x$ close to $0$. Consider the vector field, defined for $x$ close to $0$,  
	\[ X=\sum_{ij} (A(x) B(x)^{-1})_{ij} x_i \f{\p}{\p x_j}.\]
	Since $A B^{-1}$ is the identity matrix up to higher order terms, 
	this vector field is Euler-like, hence it determines a germ of a diffeomorphism $\varphi$ with $\varphi^*X=\E$. On the other hand, $X$ satisfies $\L_X f=2 f$. Pulling back under $\varphi$, we obtain $\L_\E\varphi^* f=2\varphi^*f$.  But this just means that $\varphi^*f$  is homogeneous of degree $2$, i.e., quadratic.  	Q.E.D.\bigskip

See \cite{arn:wav,ban:pro,pal:mor} for Moser-type proofs of Morse's lemma; in our approach the Moser argument is incorporated in Lemma \ref{lem:2}. Note that this argument gives a canonical choice for the 
coordinate change $\varphi$. 
\end{example}

The following proof of Darboux's theorem elaborates on a remark in the introduction of \cite{hig:eul}. 
See \cite{wei:sym} for the standard `Moser-type' proof. 

\begin{example}[Darboux theorem]
Let $\omega\in \Omega^2(\R^n)$ be a closed 2-form, which is nondegenerate at $0$. 	
Darboux's theorem states that there exists a germ of a diffeomorphism $\varphi$ such that $\varphi^*\omega$ is constant. 
Using Euler-like vector fields, this can be proved as follows.  Let $\alpha\in \Omega^1(\R^n)$ be the primitive of $\omega$ given by Poincare's lemma, and let $X\in\mf{X}(\R^n)$ be the vector field defined by $\iota_X\omega=2\alpha$.  
From the coordinate expressions
\[ \omega=\hh \sum_{ij}\omega_{ij}\d x_i \d x_j+\ldots,\ \  \alpha=\hh \sum_{ij} \omega_{ij} x_i\d x_j +\ldots,\]
(where the dots indicate terms cubic and higher in $x_i$) we see that $X$ equals the Euler vector field $\E=\sum x_i\f{\p}{\p x_i}$ up to higher order terms. Thus, $X$ is Euler-like, and determines a germ of a diffeomorphism $\varphi$ 
such that $\varphi^*X=\E$. On the other hand, pulling back the identity 
\[ \L_X\omega=\d\iota_X\omega=2\d\alpha=2\omega\]
under $\varphi$ we obtain $\L_\E\varphi^*\omega=2\varphi^*\omega$. This means that 
the 2-form $\varphi^*\omega$ is homogeneous of degree $2$, hence its coefficients are constant. Q.E.D.
\end{example}


\subsection{Weighted Euler-like vector fields on $\R^n$}
The discussion in Section \ref{subsec:lin} has the following generalization to a  \emph{weighted} (or \emph{quasi-homogeneous}) setting. 
Let $(\mathsf{w}_1,\ldots,\mathsf{w}_n)$ with $\mathsf{w}_i\in \N$ be a `weight sequence', and consider the corresponding weighted scalar multiplication 
\begin{equation}
\kappa_t\colon \R^n\to \R^n,\ (x_1,\ldots,x_n)\mapsto (t^{\mathsf{w}_1}x_1,\ldots,t^{\mathsf{w}_n}x_n).
\end{equation}
Then $s\mapsto \kappa_{\exp(-s)}$ is the flow of the weighted Euler vector field
\[ \E=\sum_{i=1}^n\, \mathsf{w}_i x_i \f{\p}{\p x^i}\]
A vector field $X=\sum_i a_i(x)\f{\p}{\p x_i}$ will be called \emph{weighted Euler-like} if  $\lim_{t\to 0} \kappa_t^*X=\E$. 
Equivalently, $\lim_{t\to 0} t^{-\mathsf{w}_i} a_i(\kappa_t(x))=\mathsf{w}_i x_i$ for $i=1,\ldots,n$.    We think of $X$ as a perturbation of $\E$ by higher order terms, where the `order' is the weighted order of vanishing. 

\begin{example} Consider $\R^2$ with weights $(\mathsf{w}_1,\mathsf{w}_2)=(1,2)$.  The vector field 
	\begin{equation}\label{eq:a} (x+y^m)\f{\p}{\p x}+2 y\f{\p}{\p y}\end{equation}
is weighted Euler-like for all $m\in \N$ since the perturbation term $y^m\f{\p}{\p x}$ has total weight $2m-1>0$. 
On the other hand, 
\begin{equation}\label{eq:b} x\f{\p}{\p x}+(2 y+x^m)\f{\p}{\p y}\end{equation}
is weighted Euler-like  only if $m\ge 3$, since the perturbation term $x^m\f{\p}{\p y}$ has total weight $m-2$.  Note that these 
are resonant cases from the perspective of Sternberg's theorem. 
\end{example}

\begin{lemma}\label{lem:3}
	If $X\in \mf{X}(\R^n)$ is weighted Euler-like for some weight $(\mathsf{w}_1,\ldots,\mathsf{w}_n)$, there exists a germ at $0$ of a diffeomorphism of $\R^n$ 
	such that $\varphi(0)=0,\ D\varphi(0)=\on{id}$, and 
	\[ \varphi^* X=\E.\]	
\end{lemma}
\begin{proof}
The proof is essentially the same as for the unweighted case,  Lemma \ref{lem:2}. We observe that the family of vector fields 
\[X_t=\kappa_t^*X=\sum_{i=1}^n \frac{a_i(\kappa_t(x))}{t^{\mathsf{w}_i}}\f{\p}{\p x_i}\]
extends smoothly to $t=0$, with $X_0=\E$, and that $\f{ d X_t}{d t}=\f{1}{t}[\E,X_t]$. We obtain a time dependent vector field $W_t$, given by $W_t=\f{1}{t}(X_t-\E)$ for $t\neq 0$. Letting $\Phi_t$ be its flow, the same calculation as in the unweighted case shows that $\Phi_{-t}^* X_t$ is constant. 
Consequently, $\varphi=\Phi_{-1}$ has the   desired properties. 
\end{proof}
\begin{example}
In particular, we see that vector fields \eqref{eq:a} for $m\ge 1$ and \eqref{eq:b} for $m\ge 3$ are 
linearizable.  The coordinate changes are explicitly given by 
$\ti{x}=x+y^m/(1-2m),\ \ti{y}=y$ for \eqref{eq:a}, and $\ti{x}=x,\ \ti{y}=y+x^m/(m-2)$ 
for \eqref{eq:b}. Note that even though \eqref{eq:a} for $m=1$ is linear in the usual sense, 
 the term $y\f{\p}{\p x}$ is considered a higher order perturbation for the weighted setting, 
and Lemma \ref{lem:3} gives a diffeomorphism removing this term. On the other hand,  
\eqref{eq:b} for $m=2$ is known to be a non-linearizable vector field.  The non-linearizability may be seen by explicitly solving the corresponding ODE, and comparing the resulting phase portrait with that of the linearized ODE.
\end{example}

\begin{remark}
The diffeomorphism $\varphi$ in Lemma \ref{lem:3} is not unique: 
in the example of $\R^2$ with weights $(1,2)$, the nonlinear diffeomorphism $\psi(x,y)=(x,y+x^2)$ commutes with $\kappa_t$, hence preserves $\E$, but $D\psi(0)=\on{id}$. Thus, if $\varphi^*$ takes $X$ to $\E$, then so does 
$(\varphi\circ \psi)^*$. 

However, one finds that $\varphi$ is uniquely determined if we require that its \emph{weighted} linear approximation is the identity, in the sense that 
$\kappa_t^{-1}\circ \varphi\circ \kappa_t\to \on{id}$ for $t\to 0$. 
\end{remark}

\section{Euler-like vector fields for submanifolds}
Consider the category of manifold pairs. Objects in this category are pairs $(M,N)$ consisting of a 
 manifold $M$ with a closed submanifold $N$; the morphisms 
$\Phi\colon (M',N')\to (M,N)$ in this category are the smooth maps $\Phi\colon M'\to M$ taking $N'$ into $N$.
The \emph{normal bundle functor} associates to every object $(M,N)$ the normal bundle \[ \nu(M,N)=TM|_N/TN,\] 
and to every morphism $\Phi\colon (M',N')\to (M,N)$  the corresponding vector bundle morphism $\nu(\Phi)\colon \nu(M',N')\to \nu(M,N)$. For example, any vector field $X\in\mf{X}(M)$ that is tangent to $N$ may be seen as a morphism 
$X\colon (M,N)\to (TM,TN)$, and applying  the normal bundle functor 
\[ \nu(X)\colon \nu(M,N)\to \nu(TM,TN)\cong T\nu(M,N)\]
we obtain the linear approximation 
$X_{[0]}=\nu(X)\in \mf{X}(\nu(M,N))$.
\begin{definition}\cite{bur:spl}
A vector field	$X\in\mf{X}(M)$ is \emph{Euler-like} for $(M,N)$ if it is tangent to $N$, with 
linear approximation  $X_{[0]}=\E$ the Euler vector field of the normal bundle. 
\end{definition}

Equivalently, Euler-like vector fields are characterized by the condition that for all functions $f$ vanishing to order $k$ on $N$,
\begin{equation}\label{eq:almostlinear} \L_X f-k\, f\ \mbox{vanishes to order $k+1$ on $N$}.\end{equation} 
Actually, it suffices to check this condition for $k=1$. 

Let $O\subset \nu(M,N)$ be a star-shaped open neighborhood of the zero section (that is, $O$ is 
invariant under multiplication by scalars in $[0,1]$).  
A morphism $\varphi\colon (O,N)\to (M,N)$  is called a \emph{tubular neighborhood embedding} if $\varphi$ is an embedding 
as an open subset $\varphi(O)=U\subset M$, and the  linear approximation $\nu(\varphi)$ is the identity map on $\nu(M,N)$, using the standard identification $\nu(O,N)\cong \nu(M,N)$. (The normal bundle of the zero section of a vector bundle is canonically isomorphic to the vector bundle itself.) The tubular neighborhood embedding will be called \emph{complete} if $O=\nu(M,N)$. 

Suppose $X$ is Euler-like for $(M,N)$, and let $\Phi_s^X$ be its (local) flow.

%
\begin{theorem}\cite{bur:spl}\label{th:euler}
	An Euler-like vector field $X$ for $(M,N)$ determines a unique maximal tubular neighborhood embedding 
	\[ \nu(M,N)\supseteq O\stackrel{\varphi}{\lra}  M\] 
	such that $\varphi^*X=\E$. Its image is the 
	open subset $U$ consisting of all $m\in M$ for which the limit $\lim_{s\to \infty}\Phi_s^X(m)$
	exists and lies in $N$. 
	The construction is functorial: Given a morphism $\Phi\colon (M',N')\to (M,N)$, and $\Phi$-related Euler-like vector fields $X,X'$, the corresponding tubular neighborhood embeddings $\varphi\colon O\to M,\ \varphi'\colon O'\to M'$ satisfy
$\Phi(O')\subset O$ and 
	$\varphi'\circ \nu(\Phi)=\Phi\circ \varphi$. 
\end{theorem}
Note that if $X$ is complete, then $U$ is invariant under $\Phi_s^X$, and hence 
the tubular neighborhood embedding $\varphi$ is complete. \medskip

According to the theorem, tubular neighborhood embeddings and Euler-like vector fields are essentially the same thing. 
%
The proof of Theorem \ref{th:euler} given in \cite{bur:spl} is analogous to the proof of Lemma \ref{lem:2}. The proof was given a clear geometric interpretation in the work of  Haj-Higson \cite{hig:eul}, using deformation spaces, further details may be found in  \cite{bis:def}.

\begin{remark}
	Euler-like vector fields for submanifolds appear in the 2016 paper by Nguyen Viet Dang \cite{dan:ext}, under the name of `Euler vector fields'. Theorem \ref{th:euler} may be deduced from the local uniqueness result \cite[Theorem 1.3]{dan:ext}.
\end{remark}

 The construction of a normal form for a given geometric structure along a submanifold  is typically the construction of a tubular neighborhood embedding in which that structure has a suitable homogeneity property (typically, linear or quadratic). By the theorem, this amounts to the construction of an Euler-like vector field $X\in\mf{X}(M)$ such that the given structure has a corresponding infinitesimal homogeneity property with respect to $X$. This reformulation of the problem can be useful.  For example, given an action of a compact Lie group $G$ preserving the structure, one can average $X$ under the group action, thus obtaining a $G$-invariant Euler-like vector field. The corresponding $\varphi$ will then produce a $G$-equivariant normal form.
 In particular, one immediately gets $G$-equivariant versions of  the Morse lemma and of Darboux's theorem.
 
\begin{example} 
Suppose $Y\in\mf{X}(M)$ is a vector field tangent to $N$, and let $\nu(Y)\in\mf{X}(\nu(M,N))$ be its linear approximation. 
A \emph{linearization} of $Y$  is a tubular neighborhood embedding $\varphi$ such $\varphi^* Y=\nu(Y)$. By Theorem \ref{th:euler}, this is equivalent to the existence of an Euler-like vector field $X$ for the pair $(M,N)$ such that $[X,Y]=\L_X Y=0$ near $N$. Indeed, in the tubular neighborhood embedding determined by $X$ we then have that $\L_\E( \varphi^*Y)=
\varphi^*(\L_X Y)=
0$, which means that $\varphi^*Y$ is linear and therefore coincides with the linear approximation. For $N=\pt$, this observation was already used by Guillemin-Sternberg in their 1968 paper \cite{gui:rem}. Note that for a compact Lie group $G$ acting on $M$, preserving the submanifold $N$, if a $G$-invariant vector field $Y$ is linearizable along $N$ then it is also equivariantly linearizable.  
\end{example}

\section{Applications}
\subsection{Weinstein's  Lagrangian neighborhood theorem}\label{subsec:weinstein}
	Let $(M,\omega)$ be a symplectic manifold, and $N\subset M$ a Lagrangian submanifold. Since 
	$\omega$ pulls back to zero on $N$, it has a well-defined linear approximation 
	\[ \omega_{[1]}\in  \Omega^2_\lin(\nu(M,N)),\] 
	which is a linear 2-form on the normal bundle. Using that $N$ is Lagrangian, one finds that 
	$\omega_{[1]}$ is symplectic along the zero section, hence also on a neighborhood of the zero section, and 
	consequently everywhere (by homogeneity). 
	Weinstein's normal form theorem \cite{wei:sym1} asserts the existence of a tubular neighborhood embedding 
	$\nu(M,N)\supset O\to M$ preserving the symplectic structures. 
 For a proof using Theorem \ref{th:euler}, we may assume, by choosing an initial tubular neighborhood embedding, that $M$ is a star-shaped open neighborhood of the zero section of $N$ in $\nu(M,N)$. 
	Let $\alpha\in \Omega^1(M)$ be the primitive of $\omega$, defined by the homotopy operator: 
	\[ \alpha=\int_0^1 \f{1}{t} \kappa_t^*\iota_\E \omega\ d t,\]
	where $\kappa_t$ is scalar multiplication by $t$ on $\nu(M,N)$. 
	Then $\alpha$ pulls back to zero on $N$. Taking the linear approximation of this equation, and using $t^{-1} \kappa_t^*\omega_{[1]}=\omega_{[1]}$, we see that $\alpha_{[1]}=\iota_\E \omega_{[1]}$. 
	Define a vector field $X\in\mf{X}(M)$ by  
	\begin{equation}\label{eq:X}
\alpha=	\iota_{X}\omega.
	\end{equation}
	Then we also have $\alpha_{[1]}=\iota_{X_{[0]}}\omega_{[1]}$, and hence $X_{[0]}=\E$ (since $\omega_{[1]}$ is symplectic). 
	 It follows that $X$ is Euler-like, and defines a tubular neighborhood embedding $\varphi\colon O\to M$. Since (on a neighborhood of $N$) 
   \[ \L_\E(\varphi^*\omega)=\varphi^*(\L_X \omega)=\varphi^*(\d \iota_X \omega)=\varphi^*\d\alpha=\varphi^*\omega,\]
   we see that $\varphi^*\omega$ is linear on a neighborhood of the zero section, and hence coincides there with the linear approximation of $\omega$. 
   That is, $\varphi^*\omega=\omega_{[1]}$ near $N$. Q.E.D.\smallskip

\begin{remark} \label{rem:canonical}
	For any vector bundle $V\to N$ with a linear symplectic form $\omega\in \Omega^2(V)$, the restriction of the 2-form to a bilinear form on $TV|_N=V\oplus TN$ 
	gives a non-degenerate pairing between 
	$V$ and $TN$. The resulting vector bundle isomorphism 
	$V\to T^*N$ is symplectic, for the canonical symplectic form on the cotangent bundle. For the case at hand, it follows that 
	$\nu(M,N)$ with the symplectic form $\omega_{[1]}$ is canonically isomorphic to $T^*N$, leading to the  
	more standard formulation of Weinstein's theorem.
 \end{remark}

\subsection{Morse-Bott functions}
Let $N\subset M$ be a closed submanifold. If $f\in C^\infty(M)$ vanishes to second order along $N$, we have the quadratic approximation $f_{[2]}\in C^\infty(\nu(M,N))$. That is, $f_{[2]}$ is fiberwise a homogeneous polynomial of degree $2$. 
The corresponding symmetric bilinear form  on the fibers $\nu(M,N)_x$ is the \emph{normal Hessian}  
\[ \on{Hess}(f)_x\colon \nu(M,N)_x\times \nu(M,N)_x\to \R.\]
The function $f$ is called \emph{Morse-Bott} along $N$ if the normal Hessian is non-degenerate for all $x\in N$.  The Morse-Bott Lemma states that in this case, there exists a tubular neighborhood embedding $\varphi\colon \nu(M,N)\supseteq O\to M$ such that $\varphi^* f=f_{[2]}$. 
By the usual strategy, this is proved once we have an Euler-like vector field for $N$ satisfying 
\begin{equation}\label{eq:quadratic} \L_X f=2f\end{equation} 
near $N$. Indeed, in the resulting tubular neighborhood embedding we have $\varphi^*X=\E$, hence $\L_\E \varphi^*f=\phi^*\L_X f=2\varphi^*f$, which means that $\varphi^*f$ is homogeneous of degree $2$ and hence coincides with its quadratic approximation $f_{[2]}$. To construct $X$, we may, for example, use an initial tubular neighborhood embedding to reduce to the 
case that $M$ is a star-shaped open neighborhood of $N\subset \nu(M,N)$, and then apply the construction from 
Example \ref{ex:morse} fiberwise.

\subsection{Grabowski-Rodkievicz theorem}
Let $E\to M$ be a vector bundle. A theorem of Grabowski-Rodkievicz \cite{gra:hig} asserts that a closed submanifold $F\subset E$ is a vector subbundle of $E$ if and only if it is invariant under the scalar multiplication $\kappa_t$ for $E$:
\begin{equation}\label{eq:Finvariant} \kappa_t(F)\subset F,\ \ \ t\in\R. \end{equation}
Using Euler-like vector fields, this may be seen as follows. Note first that the intersection $N=F\cap M$ 
is a submanifold of $F$, since it is the range of the smooth projection $\kappa_0|_F\colon F\to F$. (See, for example, \cite[Theorem 1.13]{kol:nat}.) Applying the normal bundle functor to the morphism $(F,N)\to (E,M)$ given by the inclusion of $F$,  we obtain an inclusion $\nu(F,N)\hra \nu(E,M)$ as a vector subbundle along $N$. Let $X$ be the Euler vector field of $E$. It defines a (complete) tubular neighborhood embedding $\varphi_X\colon \nu(E,M)\to E$. The invariance \eqref{eq:Finvariant} implies that $X$ is tangent to $F$, so the restriction $Y=X|_F$ is well-defined. Its linear approximation $\nu(Y)\in \mf{X}(\nu(F,N))$ is the restriction of $\nu(X)\in \mf{X}(\nu(E,M))$, so it is the Euler vector field on $\nu(F,N)$. In particular, $Y$ is Euler-like for $(F,N)$. 
The functoriality statement of Theorem \ref{th:euler} shows that the tubular neighborhood embeddings $\varphi_X,\varphi_Y$ defined by $X,Y$ fit into a commutative diagram:
\[ \xymatrix@C=10ex{\nu(F,N)\ar[r]_{\varphi_Y} \ar[d] & F\ar[d]\\ \nu(E,M)\ar[r]_{\varphi_X} & E}\]  
The map $\varphi_Y$ is a bijection, by equivariance with respect to scalar multiplication.  
Since  $\varphi_X$ is just the standard identification, and in particular is an isomorphism of vector bundles over $M$, and since $\nu(F,N)$ is a subbundle of $\nu(E,M)$ along $N$, it follows that its image $F$ is a vector subbundle of $E$ along $N$. Q.E.D.

\subsection{Linearization of proper Lie groupoids} \label{subsec:proper}
Let us next turn to linearization questions for Lie groupoids $G\rra M$. For basic information on Lie groupoids, we refer to \cite{moe:fol}.  We will denote the source and target maps by $\sz,\tz\colon G\to M$, respectively; the groupoid multiplication of elements $g,h\in G$ with $\sz(g)=\tz(h)$ is written as $g\circ h$ or simply $gh$. The Lie algebroid of $G$ is the normal bundle 
$A=\nu(G,M)$; the isomorphism $\ker(T\tz)\cong \sz^*A$ (`left-trivialization') identifies the left-invariant vector fields on $G$ with the sections of $A$. One knows that  
the infinitesimal automorphisms of a Lie groupoid $G\rra M$ are exactly the vector fields $X\in\mf{X}(G)$ that are 
\emph{multiplicative}: For all composable elements $g,h\in G$ the tangent vectors $X_g,\,X_h$ are composable elements of the tangent groupoid $TG\rra TM$, and 
\[ X_g\circ X_h=X_{gh}.\]
Equivalently, $X$ is multiplicative if and only if the vector field $(X,X)$ on $G^2$ restricts to a vector field $X^{(2)}$ on the submanifold $G^{(2)}$ 
of composable arrows, with $X^{(2)}\sim_{\Mult_G} X$ under the groupoid multiplication $\Mult_G\colon G^{(2)}\to G$.
Multiplicative vector fields are tangent to $M\subset G$, with a restriction $X_M\in\mf{X}(M)$, and 
they satisfy $X\sim_\sz X_M$ and $X\sim_{\tz} X_M$.

Suppose $N\subset M$ is an  invariant closed submanifold, that is, $\sz^{-1}(N)=\tz^{-1}(N)$. (The invariant submanifolds are unions of orbits of $G$.)  Then $G|_N=\sz^{-1}(N)$ is a closed subgroupoid of $G$. The normal bundle of $G|_N$ inside $G$  is a Lie groupoid $ \nu(G,G|_N)\rra \nu(M,N)$,  with structure maps obtained by applying the normal bundle functor to the structure maps of $G$. 
We think of $\nu(G,G|_N)$ as the linear approximation of $G$ along $N$. The Lie groupoid  $G$ is called 
\emph{linearizable along $N$} if there exists a tubular neighborhood embedding 
\[ \nu(G,G|_N)\supset O\stackrel{\varphi}{\ra} G,\]
such that $O$ is an open subgroupoid containing $N$, and $\varphi$ is a groupoid morphism. 

\begin{lemma}\label{lem:mul}
Suppose $X\in \mf{X}(G)$ is a multiplicative Euler-like vector field for $(G,G|_N)$. Then the tubular neighborhood embedding $\varphi\colon \nu(G,G|_N)\supset O\to G$ defined by $X$ is a groupoid isomorphism onto its image.
\end{lemma}
\begin{proof}
The vector field $(X,X)$ on $G^2$ is Euler-like for $(G^2,(G|_N)^2)$, hence its restriction $X^{(2)}$ to $G^{(2)}$ is Euler-like 
for $(G^{(2)},(G|_N)^{(2)})$, defining a tubular neighborhood embedding $\varphi^{(2)}\colon O^{(2)}\to G^{(2)}$ where 
$O^{(2)}=O^2\cap G^{(2)}$. Since $X^{(2)}\sim_{\Mult_G} X$, it follows by the functoriality statement of Theorem \ref{th:euler} 
that $\Mult_G(O^{(2)})\subset O$, and the diagram 
\[ \xymatrix@C=13ex{ { \nu(G^{(2)},(G|_N)^{(2)})} \supset O^{(2)}  \ar[r]_{\varphi^{(2)}} \ar[d]_{\nu(\Mult_G)} & G^{(2)}\ar[d]^{\Mult_G}\\ \nu(G,G|_N)\supset O\ar[r]_{\varphi} & G}\]  
commutes. 
 
By definition of the groupoid structure on $\nu(G,G|_N)$, the identification of 
 $\nu(G^{(2)},(G|_N)^{(2)})$ with $\nu(G,G|_N)^{(2)}$, intertwines $\nu(\Mult_G)$ with 
$\Mult_{\nu(G,G|_N)}$. Hence, the diagram tells us that $O\subset \nu(G,G|_N)$ is a subgroupoid, with $\varphi$ a   
groupoid morphism. 
\end{proof}

In general, Lie groupoids are not linearizable. Consider however the case that  the Lie groupoid $G$ is  \emph{proper}, i.e., the map 
\[ (\tz,\sz)\colon G\to \on{Pair}(M)=M\times M\] 
is a proper map. Proper Lie groupoids may be seen as counterparts to compact Lie groups; examples  include the action Lie groupoids for proper Lie group actions, or the pair groupoids for surjective submersions. See \cite{cra:dif,hoy:lie} for further information on proper Lie groupoids.  In \cite{wei:lin}, Weinstein conjectured that proper Lie groupoids are linearizable at their critical points; this conjecture and its generalization to linearization along orbits
was proved by Zung in \cite{zun:pro}. (The orbits $\O\subset M$ of a proper Lie groupoid are invariant closed submanifolds of $M$.) Simpler arguments (and improved statements) were given later by Crainic-Struchiner \cite{cra:lin} and Fernandes-del Hoyo \cite{fer:rie}. 
We will explain how the proof in \cite{cra:lin} can be further streamlined using Euler-like vector fields. A key tool for proper Lie groupoids is the existence of `Haar densities'. A \emph{left-invariant density} $\m$ on a Lie groupoid $G$ is a left-invariant section of the density bundle of $\ker(T\tz)\subset TG$; under left trivialization it is  
identified with a section of the density bundle of $A$. The following fact is due to J.-L. Tu \cite{tu:con} and Crainic \cite[Section 2.1]{cra:dif}:
\begin{lemma}\label{lem:leftinv}
	Proper Lie groupoids $G\rra M$ admit left-invariant 
	densities $\m$ such that  
	\begin{itemize} 
		\item the restriction of $\tz\colon G\to M$ to the support of $\m$ 
		is a proper map,
		\item the integral of $\m$ over every $\tz$-fiber is equal to $1$.	
	\end{itemize}
\end{lemma}
These `Haar densities' are used for averaging arguments. In particular, as shown in \cite{cra:lin}, they may be 
used to obtain multiplicative vector fields: 
\begin{lemma}[Crainic-Struchiner \cite{cra:lin}] Let $G\rra M$ be a proper Lie groupoid. Then any 
	left-invariant density $\m$ as in Lemma \ref{lem:leftinv} determines a projection from the space of $\sz$-projectable vector fields 
on $G$ to the space multiplicative vector fields, given by $X\mapsto  \ol{X}$ with
\[ 	 \ol{X}_g=\int_{a\in \tz^{-1}(\sz(g))} X_{ga}\circ (X_a)^{-1} \m(a)\]	
(using groupoid multiplication in $TG\rra TM$). The averaged vector field satisfies  $\ol{X}\sim_{\sz} \ol{Y}$ where 
\[ \ol{Y}_m=\int_{a\in \tz^{-1}(m)} \tz(X_a)\ \m(a).\]
\end{lemma}
In this lemma, the condition that  $X$ is \emph{$\sz$-projectable} (i.e.,  $X\sim_\sz Y$ for some vector field $Y\in\mf{X}(M)$) ensures that the groupoid product $X_{ga}\circ (X_a)^{-1} $ is defined.

\begin{theorem} \label{th:grlin}
	If $G\rra M$ is a proper Lie groupoid, and $N\subset M$ is a closed invariant submanifold, then $G$ is linearizable along $N$. 	
\end{theorem}
\begin{proof}
By Lemma  \ref{lem:mul}, it suffices to construct a multiplicative Euler-like vector field for $(G,G|_N)$. We begin by choosing 
an Euler-like vector field $Y$ for $(M,N)$. Since $\sz$ is a submersion, there exists  $X\in\mf{X}(G)$ with $X\sim_{\sz} Y$. Since $\sz$ induces a fiberwise isomorphism of normal bundles, 
the vector field $X$ is Euler-like for $(G,G|_N)$. Let $\ol{X}$ be the multiplicative vector field obtained by taking the average with respect to a Haar density $\m$. We claim that the vector field $\ol{X}$ is again Euler-like. Since $\ol{X}\sim_\sz \ol{Y}$, it suffices to show that $\ol{Y}$ is Euler-like. We shall use the characterization \eqref{eq:almostlinear}. Let $f\in C^\infty(M)$ 
with $f|_N=0$. Since $X$ is Euler-like, the difference $\L_X (\tz^*f)-\tz^* f$ vanishes to second order on $G|_N\subset G$. But then its integral over $\tz$-fibers 
\[ m\mapsto (\L_{\ol{Y}}f-f)(m)=\int_{a\in\tz^{-1}(m)} \big(\L_X (\tz^* f)-\tz^* f\big)\Big|_a\ \m(a)\]
vanishes to second order on $N$, as required.  
\end{proof}
\vskip.3in
\subsection{Linearization of proper symplectic groupoids}
It is possible to combine the Euler-like arguments for the linearization of proper Lie groupoids (Section \ref{subsec:proper}) and for Weinstein's Lagrangian neighborhood theorem (Section \ref{subsec:weinstein}). Recall that a differential 
 form $\alpha\in \Omega(G)$ on a groupoid $G\rra M$ is \emph{multiplicative} if  $\Mult_G^*\alpha=\pr_1^*\alpha+\pr_2^*\alpha$ where $\Mult_G\colon G^{(2)}\to G$ is the groupoid multiplication and $\pr_1,\pr_2$ are the projections to the two factors of 
$G^{(2)}\subset G^2$. A \emph{symplectic groupoid} is a groupoid with a 
multiplicative symplectic form $\omega\in \Omega^2(G)$. 

Given an invariant closed submanifold $N\subset M$, the restriction $G|_N$ is a coisotropic submanifold; it is Lagrangian if and only if $\dim N=0$. Thus suppose $N=\{p\}$ is a critical point for $G$. In this case, $H=G|_{\{p\}}$ (as a groupoid over a point) is a Lie group. The normal bundle inherits the structure of a Lie groupoid over $\nu(G,H)\rra \nu(M,\pt)=T_pM$, equipped with a linearized symplectic form $\omega_{[1]}\in \Omega^2(\nu(G,H))$, which is again multiplicative. 

Suppose that the symplectic groupoid $G\rra M$ is furthermore \emph{proper}, so that $H$ is, in particular, compact. 
We claim that $\omega$ admits a multiplicative primitive $\alpha\in \Omega^1(G)$ (defined on some neighborhood of $H$ in 
$G$), with $\alpha|_N=0$.  Indeed, by Theorem \ref{th:grlin} there exists a  tubular neighborhood embedding 
\[ \nu(G,H)\supset O_1\stackrel{\varphi_1}{\lra} G\] that is compatible with the groupoid structures. Thus $\varphi_1^*\omega$ is multiplicative, and pulls back to $0$ on $H$. By applying the standard homotopy operator  for the retraction of $O_1$ onto $N$, we obtain a primitive for $\varphi_1^*\omega$, which is furthermore multiplicative. 
Write this primitive  in the form $\varphi_1^*\alpha$ (near $N$), and let $X$ be defined by $\iota_X\omega=\alpha$  (as in Section \ref{subsec:weinstein}). Then $X$ is Euler-like, with $\L_X\omega=\omega$ near $H$, and is also  multiplicative (since both $\omega$ and $\alpha$ are multiplicative). Hence, the tubular neighborhood embedding 
 \[ \nu(G,H)\supset O\stackrel{\varphi}{\lra} G\] 
defined by $X$ is an isomorphism of symplectic groupoids. The canonical isomorphism $\nu(G,H)\cong T^*H$ from Remark \ref{rem:canonical} is an isomorphism of symplectic groupoids. (See \cite[Proposition 11.14]{me:intropoi}.) We hence recover Zung's result \cite[Theorem 2.5]{zun:pro} that a proper symplectic groupoid $G\rra M$ is isomorphic, near any critical point $p\in M$, to a cotangent groupoid $T^*H\rra \h^*$.

\section{The weighted setting}
In this section we indicate a generalization of Euler-like vector fields for pairs $(M,N)$ to include \emph{weightings}. Further details will be given in a 
forthcoming work with Y.~Loizides \cite{meloi:prep}. For simplicity, we will restrict ourselves to
the degree 2 case (the filtration on functions by `weighted degree of vanishing' is determined by its components up to filtration degree $2$); the general case is slightly more involved and will be discussed in  \cite{meloi:prep}.

\subsection{The filtration on functions}\label{subsec:filtration}
Suppose $N\subset M$ is a closed submanifold. 
Denote by  $\A=C^\infty_M$ the sheaf of smooth functions on $M$, and by 
$\I\subset \A$ the vanishing ideal (ideal sheaf) of $N$. Then $\I^k$ is the sheaf of functions vanishing to order $k$ on $N$. Given a subbundle $F\subset \nu(M,N)$, we may define a smaller ideal $\J\subset \I$, where $\J(U)$ consists of 
functions $f\in C^\infty(U)$ satisfying $f|_U=0,\ \d f|_{TU\cap \ti{F}}=0$. (Here $\ti{F}\subset TM|_N$ is the pre-image 
of $F$, and the differential $\d f$ is thought of as a function on the tangent bundle.)  Note $\I^2\subset \J$. The ideals $\I,\J$ generate a decreasing filtration 
of the sheaf of algebras algebras $\A=C^\infty_M$, 
\[ C^\infty_M=\A_{(0)}\supset \A_{(1)}\supset \cdots \]
with $\A_{(1)}=\I,\ \A_{(2)}=\J$, and with $\A_{(k)}(U)$ for $k>1$ spanned by products of functions $f_1\cdots f_p\,g_1\cdots g_q$ with $f_i\in \I,\ g_j\in \J$, and $p+2q\ge k$. Equivalently, for $\ell\in\N$, 
\[ \A_{(2\ell-1)}=\I\J^\ell,\ \A_{(2\ell)}=\J^\ell.\]
By construction, $\A_{(k)}\A_{(k')}\subset \A_{(k+k')}$ for all $k,k'$.

\begin{remark}\label{rem:feeling}
To get a better feeling for this filtration, choose local coordinates $x_a,y_b,z_c$ 
on an open neighborhood $U$ of any given point $p\in N$, in such a way that $U\cap N$ is defined by the vanishing of the coordinates $y_b,z_c$, and 
$\ti{F}\cap TU$ is defined by the vanishing of $y_b,z_c$, together with the vanishing of the differentials $\d z_c$. Then $\I(U)$ is the ideal generated by the functions $y_b,\,z_c$, while $\J(U)$ is generated by the functions $y_{b_1}y_{b_2}$ and $z_c$. Consequently, 
\[ \A_{(k)}(U)=\langle y_{b_1}\cdots y_{b_p}z_{c_1}\cdots z_{c_q}|\ p+2q= k\rangle\]
(where $\langle\cdots\rangle$ denotes the ideal in $\A(U)$ generated by this subset).  
This amounts to assigning weight $2$ 
to the coordinates $z_c$, weight $1$ to the coordinates $y_b$, and weight $0$ to the coordinates $x_a$. 
\end{remark}
The filtration of the sheaf of algebras $\A$ gives an associated graded sheaf of algebras,  
\[ \on{gr}(\A)=\bigoplus_{k=0}^\infty \on{gr}(\A)_k,\ \ \ \ 
\on{gr}(\A)_k= \A_{(k)}/\A_{(k+1)}\] 
By construction, it is generated by its components of degree $\le 2$. 
Note that $\on{gr}(\A)$ is supported on $N$, since $\A_{(k)}(U)=\A$ for $U\cap N=\emptyset$, and hence may be regarded as a sheaf on $N$. 
\begin{lemma}
The sheaf $\on{gr}(\A)$ is the sheaf of sections of  a bundle of graded algebras $\mathsf{A}\to N$. 
\end{lemma}
\begin{proof}
This is clear in the local coordinates from Remark \ref{rem:feeling}: Here $\on{gr}(\A)_k(U)$ is the 
$C^\infty(U\cap N)$-span of monomials  $y_{b_1}\cdots y_{b_p}z_{c_1}\cdots z_{c_q}$ with $p+2q=k$.
Put differently, $\on{gr}(\A)(U)$ is generated (as an algebra over $C^\infty(U\cap N)$) by the coordinate functions $y_b$ 
(of degree $1$) and $z_c$ (of degree $2$).  
\end{proof}

\begin{proposition}\label{prop:a}
	There are canonical isomorphisms
	\[ \mathsf{A}_0=
	N,\ \ \ \mathsf{A}_1=F^*,\] 
	and an exact sequence 
	\begin{equation}\label{eq:exact} 0\to \on{Sym}^2(F^*)\to \mathsf{A}_2\to  \on{ann}(F)\to 0.\end{equation}
\end{proposition}	
\begin{proof}
	The description of $\sA_0$ is clear. The identification of 
	$\sA_1$ with $F^*=\nu(M,N)^*/\on{ann}(F)$ follows from the equality of the sheaves of sections,  
		\[ \on{gr}(\A)_1=\I/\J=(\I/\I^2)\big/(\J/\I^2).\]
	Finally, the quotient map $\on{gr}(\A)_2=\J/\I\J\to  \J/\I^2$ gives a surjection 
	$\sA_2\to \on{ann}(F)$. The kernel 
	$\I^2/\I\J$ is identified with the sheaf of sections of 
	$\on{Sym}^2(F^*)$ through the multiplication map  
	$\I/\J\otimes \I/\J\to \I^2/\I\J$.
\end{proof}
	
\begin{proposition}\label{prop:split}	The choice of a submanifold $\Sigma$ with 	\[ N\subset \Sigma\subset M,\ \ \ F=\nu(\Sigma,N)\subset \nu(M,N)\] determines a splitting of the exact sequence \eqref{eq:exact},  	and gives an isomorphism
of algebra bundles
\[ 	\sA\cong \on{Sym}\big(F^*\oplus \on{ann}(F)\big).\]
with $F^*$ in degree $1$ and $\on{ann}(F)$ in degree $2$. 
\end{proposition}	

\begin{proof} Let $\I_\Sigma\subset \J$ be the ideal of germs of functions $f$ such that $f|_\Sigma$ vanishes near $N\subset \Sigma$. We claim that	\begin{equation}\label{eq:jsigma} \J=\I_\Sigma+\I^2,\ \ \ \ \I_\Sigma\cap \I^2=\I\I_\Sigma=\I_\Sigma\cap \I\J.\end{equation}	To see this, given $p\in N$, choose coordinates $x_a,y_b,z_c$  	so that the coordinate functions $y_b,z_c$ vanish on $N$, and $z_c$ vanishes on $\Sigma$.  Then 
	 $\I(U)$ is generated by the coordinate functions $y_b,z_c$, while $\J(U)$ is generated by $z_c$ and products $y_{b_1}y_{b_2}$. 	The smaller ideal $\I_\Sigma(U)$ is generated by $z_k$, and   each of 	$\I_\Sigma\cap \I^2,\ \I\I_\Sigma,\ \I_\Sigma\cap \I\J $ is generated over $U$  by products $y_j z_k$. By \eqref{eq:jsigma}, the inclusion $\I_\Sigma\hra \J$ descends to  an inclusion	\[ \J/\I^2\cong \I_\Sigma/(\I_\Sigma\cap \I^2)=\I_\Sigma/(\I_\Sigma\cap\I\J)\hra \J/\I\J,\]	which defines the desired splitting $\on{ann}(F)\hra \sA_2$. The inclusions of 	$F^*=\sA_1$ and   $\on{ann}(F)\subset \sA_2$ into $\sA$ extend to a morphism of graded algebra bundles	
	 \begin{equation}\label{eq:algmor} \on{Sym}(F^*\oplus \on{ann}(F))\to \sA.\end{equation}	
	 Since $\sA$ is generated by its components in degree $\le 2$, this map is surjective. Using local coordinates $x_a,y_b,z_c$ as above, we see that it is in fact an isomorphism. 
	 \end{proof}
 \begin{remark}
 	The subbundle $F$ may be regarded as a (very short) filtration of the normal bundle $\nu(M,N)$, with associated graded bundle 
 	\[\on{gr}(\nu(M,N))=F\oplus \nu(M,N)/F.\] Using the identification $(\nu(M,N)/F)^*\cong \on{ann}(F)$, we may thus regard the algebra bundle $\on{Sym}\big(F^*\oplus \on{ann}(F)\big)$ as the bundle of fiberwise polynomial functions on $\on{gr}(\nu(M,N))$. 
 \end{remark}
 
 \begin{remark}The splitting in Proposition \ref{prop:split} only depends on the 2-jet of $\Sigma$ along $N$, i.e., on the restriction $T_2\Sigma|_N\subset T_2M=J^2_0(\R,M)$ of the second-order tangent bundle of $\Sigma$. \end{remark}

\subsection{The weighted normal bundle}\label{subsec:weight}
Recall \cite{gon:dif} (see also \cite{hig:eul}) that the smooth manifold $M$ is determined from 
its algebra of functions as $M=\on{Hom}_{\on{alg}}(C^\infty(M),\R)$, where the algebra morphism defined by 
$p\in M$ is given by evaluation $\on{ev}_p\colon C^\infty(M)\to \R$. The smooth structure of $M$ is determined
by requiring that all evaluation maps are smooth. For a vector bundle $E\to M$, one similarly has 
$E=\on{Hom}_{\on{alg}}(C^\infty_{\on{{pol}}}(E),\R)$. The scalar multiplication on $E$ (characterizing the vector bundle structure) comes from the usual $(\R,\cdot)$-action on the space $C^\infty_{\on{{pol}}}(E)$ of fiberwise polynomial functions, and the base projection comes from the inclusion of degree $0$ polynomials (identified with $C^\infty(M)$).

For a closed submanifold $N\subset M$, the algebra of polynomial functions on the normal bundle $\nu(M,N)$ is identified with the associated graded algebra to $C^\infty(M)$, using the filtration by order of vanishing on $N$. That is, 
\[ \nu(M,N)=\on{Hom}_{\on{alg}}(\on{gr}(C^\infty(M)),\R).\]
Given a subbundle $F\subset \nu(M,N)$, we have another filtration of $C^\infty(M)$ by the \emph{weighted} order of vanishing, as in Section \ref{subsec:filtration} above.   We shall define a weighted normal bundle similar to the usual normal bundle, but using this new filtration. As it turns out, the resulting space $\nu_\W(M,N)$ is a smooth fiber bundle over $N$. 
While $\nu_\W(M,N)$ is not a vector bundle, it does have an action of the monoid $(\R,\cdot)$ by scalar multiplication, making it into a \emph{graded bundle} in the sense of Grabowski-Rodkievicz. In particular, the associated graded algebra
to $C^\infty(M)$ is then identified with the polynomial functions on $\nu_\W(M,N)$. 
\begin{definition}
Let $(M,N)$ be a manifold pair and $F\subset \nu(M,N)$ a subbundle, defining a filtration of $C^\infty(M)$ as above.  	
	The \emph{weighted normal bundle} $\nu_\W(M,N)$ is the character spectrum 
\[ \nu_\W(M,N)=\on{Hom}_{\on{alg}}(\on{gr}(C^\infty(M)),\R),\]
with the projection $\nu_\W(M,N)\to N$ induced by the inclusion $C^\infty(N)\to \on{gr}(C^\infty(M))$ as the degree $0$ summand. 

\end{definition}
\begin{remark}
In this form, the definition is specific to the $C^\infty$-category. For more general settings (including the 
analytic or holomorphic categories), one should work with the filtration of the sheaf $C^\infty_M$ rather than of global functions. The definition above may then be used locally to define $\nu_\W(M,N)$ locally; one may also 
directly define the fibers in terms of the algebra bundle $\sA$, as  
\[\nu_\W(M,N)|_p=\on{Hom}_{\on{alg}}(\sA|_p,\R).\]
\end{remark}
\bigskip

To describe the smooth structure on $\nu_\W(M,N)$, choose a submanifold  $\Sigma$ containing $N$, with the property
\[ \nu(\Sigma,N)=F.\]
Proposition  \ref{prop:split} identifies the graded algebra bundle
$\mathsf{A}\to N$ with the  bundle of polynomial functions on the fibers of the graded bundle $\on{gr}(\nu(M,N))$, hence 
$\on{gr}(C^\infty(M))\cong C^\infty_{\on{pol}}(\on{gr}(\nu(M,N))$. 
Hence, the choice of any such $\Sigma$ determines an identification 
\begin{equation}\label{eq:rsigma}
R_\Sigma\colon  \nu_\W(M,N)\cong \on{gr}(\nu(M,N)),\end{equation}
intertwining the projection $\pi$ with the vector bundle projection for $\on{gr}(\nu(M,N))$. 
The resulting  $C^\infty$ structure on $\nu_\W(M,N)$ does not depend on the choice of $\Sigma$. To see this, 
note that every $a\in \on{gr}(C^\infty(M))$  defines a function on the weighted normal bundle, by evaluation: 
\begin{equation} \label{eq:evaluation}
\on{ev}_a\colon \nu_\W(M,N)\to \R,\ x\mapsto x(a).\end{equation}
The $C^\infty$ structure on $\nu_\W(M,N)$ is uniquely determined by requiring that all of these evaluation maps are smooth. 
\medskip


Returning to the general definition, we note the following properties:

\begin{enumerate}
	\item The weighted normal bundle comes with a natural `zero section' 
	\[  \iota\colon N\to \nu_\W(M,N),\]
	with $\pi\circ \iota=\on{id}_N$, 
	obtained by applying the functor $\on{Hom}_{\on{alg}}(\cdot,\R)$ to the augmentation map 
	$\on{gr}(C^\infty(M))\to \on{gr}(C^\infty(M))_0=C^\infty(N)$. 
	Under $R_\Sigma$, this becomes the  zero section of $\on{gr}(\nu(M,N))$. 
	\item\label{it:c} The action of the multiplicative monoid $(\R,\cdot)$ on $\on{gr}(\A)$, given on $\on{gr}(\A)_k$ as multiplication by $t^k$, is by algebra bundle morphisms. Hence it induces an action $t\mapsto \kappa_t$ on $\nu_\W(M,N)$ such that 
	\[ \on{ev}_{t\cdot a}=\on{ev}_a\circ \kappa_t,\ \ a\in \on{gr}(\A).\]
	Under $R_\Sigma$, this is the action on $\on{gr}(\nu(M,N))$ given as scalar multiplication by $t$ on $F$ and by $t^2$ on $\nu(M,N)/F$. In particular, the $(\R,\cdot)$-action on the weighted normal bundle is smooth. It makes  
	$\nu_\W(M,N)$ into a \emph{graded bundle} in the sense of Grabowski-Rodkievicz \cite{gra:gra}. (See below.) 
	\item\label{it:cd} Consider the morphisms of algebra bundles
	\[ \sA'\lra \sA\lra \sA'',\] 
	where $\sA'\subset \sA$ is the subalgebra bundle generated by $\sA_1$, while $\sA''=\sA$ is the quotient of $\sA$ by the ideal generated by $\sA_1$. Thus, $\sA'=\on{Sym}(F^*)$, the algebra of polynomial functions on $F$, while 
	$\sA''=\on{Sym}(\on{ann}(F))$, the algebra of polynomial functions on $\nu(M,N)/F$. 
	Applying the functor $\on{Hom}_{\on{alg}}(\cdot,\R)$, we obtain $(\R,\cdot)$-equivariant bundle maps 
	\[  \nu(M,N)/F\lra \nu_\W(M,N)\lra F.\]	
	Under the map $R_\Sigma$, these become the obvious inclusion and projection $\nu(M,N)/F\to \on{gr}(\nu(M,N))\to F$.  
	 \item \label{it:functorial}
   The construction is functorial: Consider the category whose objects are triples $(M,N,F)$ consisting of a manifold pair $(M,N)$ and a subbundle $F\subset \nu(M,N)$, and whose morphisms $(M',N',F')\to (M,N,F)$ are the morphisms of manifold pairs $\varphi\colon (M',N')\to (M,N)$ such that $\nu(\varphi)$  takes $F'$ to $F$. For any such morphism, the 
   pullback map on functions 
   $C^\infty(M)\to C^\infty(M')$ preserves the weight filtrations; it therefore  induces a morphism of graded algebras,  
   $\sA\to \sA'$, and via  $\on{Hom}_{\on{alg}}(\cdot,\R)$ a morphism of graded bundles 
   \[ \nu_\W(\varphi)\colon \nu_\W(M',N')\to \nu_\W(M,N).\] 
 \end{enumerate}

Following Grabowski-Rodkievicz \cite{gra:gra}, a manifold $P$ with a smooth action 
\[ \R\times P\to P,\ (t,x)\mapsto \kappa_t(x)\]  
of the monoid $(\R,\cdot)$ is called a \emph{graded bundle}. As shown in \cite{gra:gra}, 
these are automatically smooth fiber bundles over the submanifold $N=\kappa_0(P)$, with $\kappa_0$ as the bundle projection. 
First examples of graded bundles are the graded \emph{vector} bundles, with the $(\R,\cdot)$-action given as multiplication by  $t^k$ on the $k$-th summand, as well as the \emph{$r$-th tangent bundle} 
$T_rM=J_0^r(\R,M)$, with the $(\R,\cdot)$-action by linear reparametrizations of curves.
According to \eqref{it:c} above, the weighted normal bundle is a graded bundle. 
Every graded bundle $P$ comes with a distinguished \emph{Euler vector field} 
\[ \E\in\mf{X}(P)\]  with flow  given by $s\mapsto \kappa_{\exp(-s)}$. Morphisms of graded bundles are the $(\R,\cdot)$-equivariant smooth maps; equivalently these are the maps relating the Euler vector fields. 
A graded bundle $P$ is not naturally a vector bundle, in general. However, it determines a graded vector bundle $P_\lin= \nu(P,N)$, the \emph{linear approximation}, to which it is non-canonically isomorphic, as a graded bundle. 
In the case of $\nu_\W(M,N)$, this graded vector bundle only has components in degree $1$ and $2$, and is 
\begin{equation}\label{eq:nunu} \nu_\W(M,N)_\lin=\on{gr}(\nu(M,N)).\end{equation}
The map $R_\Sigma\colon \nu_\W(M,N)\to \on{gr}(\nu(M,N))$ is an isomorphism of graded bundles. Note finally that 
graded bundles come with a distinguished graded algebra $C^\infty_{\on{pol}}(E)$ of (fiberwise) polynomial functions, 
where a function is a homogeneous polynomial of degree $k$ if it satisfies $\L_\E f=kf$.
The description in \eqref{it:alter} may be seen as a special case of a general construction for graded bundles in terms of their polynomial functions, see \cite{gra:gra}. 

To conclude this section, let us mention two alternative constructions of the weighted normal bundle: 
\begin{enumerate}
	\item[1.] 	\label{it:alter}
	Consider the  exact sequence 
	\eqref{eq:exact}. The first map in this sequence is the map $m\colon \on{Sym}^2 \sA_1\to \sA_2$
	given by multiplication in the graded algebra bundle $\sA$. The weighted normal bundle is realized as a subbundle of the graded  vector bundle $\sA_1^*\oplus \sA_2^*\to N$, as follows:
	\[ \nu_\W(M,N)=\{(\alpha_1,\alpha_2)\in \sA_1^*\oplus \sA_2^*|\ \forall \mu\in \sA_1\colon \alpha_2(m(\mu\vee \mu))=\alpha_1(\mu)^2\}.\]
	(Here it is understood that $\mu$ has the same base point as $(\alpha_1,\alpha_2)$.)
	Note that this subset is invariant under the $(\R,\cdot)$-action, where $t\in \R$ acts as scalar multiplication by $t$ on $\sA_1^*$ and by $t^2$ on $\sA_2^*$. 
	Using a choice of $\Sigma$, one verifies that it is indeed a smooth submanifold.
	\item[2.] Similar to the description of the usual normal bundle as a subquotient of the tangent bundle, $\nu(M,N)=TM|_N/TN$, the weighted normal bundle admits a description as a subquotient of the second tangent bundle $T_2M=J_0^2(\R,M)$, 
		\[ \nu_W(M,N)=Q/\sim\]
	where 	$Q\subset T_2M$ is the pre-image of $\ti{F}\subset TM|_N$ under the projection $T_2M\to TM$. 
	See 
	\cite{meloi:prep} for a description of the equivalence relation. 
\end{enumerate}

\subsection{Weighted $k$-th order approximations}
The filtration on the sheaf of functions $\A=C^\infty_M$ extends to a filtration on the sheaf of differential forms $\Omega_M$, 
\[ \Omega_M=\Omega_{M,(0)}\supset \Omega_{M,(1)}\supset \Omega_{M,(2)}\supset \cdots, \]
where the part of filtration degree $k$ is locally spanned by forms $f_0\ d f_1\cdots d f_q$ such that $f_i\in \A_{(\mathsf{w}_i)}$, with 
$\mathsf{w}_0+\ldots+\mathsf{w}_q\ge k$. Note that $\alpha\in \Omega(M)$ has filtration degree $1$ if and only if $i_N^*\alpha=0$, and 
it has filtration degree $k\ge 2$ if and only if for all $X\in\mf{X}(M)$ with $X|_N\in \Gamma(\ti{F})$, both $\iota_X\alpha,\ \iota_X\d\alpha$ have filtration degree $k-1$. One also obtains a filtration on the sheaf of vector fields, 
\[ \mf{X}_M=\mf{X}_{M,(-2)}\supset \mf{X}_{M,(-1)}\supset \mf{X}_{M,(0)}\supset \cdots\]
where a (local) vector field has filtration degree $k$ if and only if the associated Lie derivative $\L_X$ raises the filtration degree of functions by $k$.  In particular, $X$ has 
filtration degree $-1$ if and only if $X|_N\in \Gamma(\ti{F})$; it has filtration degree $0$ if and only if $X$ is tangent to $N\subset M$ and the linear approximation $\nu(X)$ is tangent to $F\subset \nu(M,N)$. Thus, vector fields of filtration degree $0$ are the infinitesimal automorphisms of the triple $(M,N,F)$.
These filtrations are compatible with the usual operations from Cartan's calculus: exterior differential,
contraction,   Lie derivative, as well as  wedge products, $C^\infty_M$-module structures,  and Lie brackets.

 As before, we will refer to the filtration degrees as \emph{weights},  and think of a function $f$ 
 of filtration degree $k$ as a  function vanishing to order $k$ on $N$ \emph{in the weighted sense}. Its equivalence class in the associated graded algebra defines a function  
\[ f_{[k]}\in C^\infty(\nu_W(M,N))\] 
that is homogeneous of degree $k$ for the $(\R,\cdot)$-action. More generally, a form $\alpha\in \Omega(M)$ of filtration degree $k$ determines a form $\alpha_{[k]}\in \Omega(\nu_\W(M,N))$, homogeneous of degree $k$, and a vector field $X\in \mf{X}(M)$ of filtration degree $k$ defines a vector field 
$X_{[k]}\in\mf{X}(\nu_\W(M,N))$, homogeneous of degree $k$. Again, these k-th order homogeneous approximations are compatible with the standard operations on vector fields and forms, for example $(\L_{X}\alpha)_{[k+\ell]}=\L_{X_{[k]}}\alpha_{[\ell]}$ if $X$ has filtration degree $k$ and $\alpha$ has filtration degree $\ell$.


\subsection{Weighted Euler-like vector fields}
As before, let $(M,N)$ be a manifold pair together with a vector subbundle 
$F\subset \nu(M,N)$, defining a weighted normal bundle $\nu_\W(M,N)$. The normal bundle of the manifold pair 
$(M',N')=(\nu_\W(M,N),N)$ 
is given by 
\[ \nu(M',N')=\nu(\nu_\W(M,N),N)\cong \on{gr}(\nu(M,N))=F\oplus \nu(M,N)/F,\]
and so again contains  $F'=F$ as a vector subbundle.  We may thus iterate the construction, and consider the corresponding weighted normal bundle. 
\begin{lemma}
	There is a canonical isomorphism of graded bundles over $N$, 
	\[\nu_\W(\nu_\W(M,N),N)\cong \nu_\W(M,N).\]
\end{lemma}
\begin{proof} The data $(M',N',F')$ define the filtration $\A'_{(k)}$ of $C^\infty_{M'}$, with the corresponding graded algebra bundle 
	$\sA'$. 
	The maps $\A_{(k)}(U)\to C^\infty(\nu_\W(M,N)|_{U\cap N}),\ f\mapsto f_{[k]}$, for $U\subset M$ open,  take values in functions
	of homogeneity $k$, hence of filtration degree $k$. This defines isomorphisms $\sA_k\to \sA_k'$ of vector bundles over $N$. 
	The collection of these maps defines an algebra bundle isomorphism $\sA\to \sA'$, and hence an isomorphism of the weighted normal bundles. 
\end{proof}

Using this observation, we can define a (weighted) \emph{tubular neighborhood embedding} of a star-shaped open neighborhood $O$ of $N\subset \nu_\W(M,N)$ to be a morphism $\varphi\colon (O,N,F)\to (M,N,F)$ such that 
\begin{equation}\label{eq:weighttub} \nu_\W(M,N)\supset O\stackrel{\varphi}{\lra} M\end{equation}
is an embedding, and $\nu_\W(\varphi)$ is the identity map. On the other hand, a vector field $X\in\mf{X}(M)$ 
 is called (weighted) \emph{Euler-like} for $(M,N,F)$ if it has filtration degree $0$ (so, it preserves $N$ and $F$), and its approximation $X_{[0]}\in\mf{X}(\nu_\W(M,N))$ is the Euler vector field. Equivalently, $X$ is Euler-like if and only if for all 
 functions $f$ of filtration degree $k$, the difference $\L_X f-k f$ has filtration degree $k+1$. (It is enough to check this condition  for $k=1,2$.)

\begin{theorem}\label{th:eulerw}
	A weighted Euler-like vector field $X$ for  $(M,N,F)$ 
	determines a unique maximal tubular neighborhood embedding \eqref{eq:weighttub},
	in such a way that $\varphi^*X=\E$. 
	The construction is functorial: Given a morphism $\Phi\colon (M',N',F')\to (M,N,F)$, 
	and weighted Euler-like vector fields $X',X$ with $X'\sim_\Phi X$, the 
	corresponding tubular neighborhood embeddings intertwine $\Phi$ with its linear approximation $\nu_\W(\Phi)$.
\end{theorem}

The proof of this result, using (weighted) deformation spaces, is a straightforward extension of the argument in 
\cite{hig:eul} and \cite{bis:def}. Details, for more general weightings, will be given in \cite{meloi:prep}. Alternatively, one may give a proof in coordinates $x_i,y_j,z_k$ adapted to $N\subset M$ and to $F\subset \nu(M,N)$, similar to the proof of  Lemma \ref{lem:3}.

\subsection{Application: the isotropic embedding theorem re-visited}
Let $(M,\omega)$ be a symplectic manifold. Recall that a closed submanifold $N\subset M$ is \emph{isotropic} if 
the pullback of $\omega$ to $N$ vanishes. Equivalently,  letting $TN^\omega\subset TM|_N$ denote the $\omega$-orthogonal space to the tangent bundle of $N$, we have that $TN\subset TN^\omega$. 
The \emph{symplectic normal bundle} of the isotropic submanifold $N$ is the subbundle  of the normal bundle given as  
\begin{equation}\label{eq:ourf} F=TN^\omega/TN\subset \nu(M,N).\end{equation}
By our general theory, it determines a weighted normal bundle  $\nu_\W(M,N)\to N$. The corresponding sequence of graded bundles from \ref{subsec:weight}\eqref{it:cd},  reads as 
\begin{equation}\label{eq:ourex} T^*N\to \nu_\W(M,N)\to TN^\omega/TN,\end{equation}
where we identified $\nu(M,N)/F=TM/TN^\omega\cong T^*N$. 
We will show that $\omega$ canonically induces a symplectic form on the weighted normal bundle, extending the 
 standard symplectic form $\omega_{\on{can}}$ on $T^*N$. 
 
\begin{proposition} For any isotropic submanifold $N$ of a symplectic manifold $(M,\omega)$, 
the symplectic form $\omega$ has filtration degree $2$ with respect to the weighting given by $F=TN^\omega/TN$. 
The second order approximation 
\[ \omega_{[2]}\in \Omega^2(\nu_\W(M,N))\] is symplectic form. Relative to this symplectic form, 
$T^*N$ is a symplectic submanifold of $\nu_W(M,N)$; in particular,  the zero section $N$ is isotropic. 
\end{proposition} 
\begin{proof} 
To show that $\omega$ has filtration degree $2$, we must show that for any vector field $X\in\mf{X}(M)$ with $X|_N=\Gamma(\ti{F})$, the 1-form $\iota_X\omega$ pulls back to zero on $N$. But this is clear since $\ti{F}=TN^\omega$, by definition. Hence the second-order approximation $\omega_{[2]}$ is defined, with $\d\omega_{[2]}=(\d\omega)_{[2]}=0$. 
Consider the decomposition 
\begin{equation}\label{eq:decomp} T\nu_\W(M,N)|_N=TN\oplus TN^\omega/TN\oplus T^*N.\end{equation}
In terms of the linearized action of the (weighted) Euler vector field, these are the eigenbundles for the 
eigenvalues $0,-1,-2$. Accordingly, sections of the three summands are realized as the restrictions of 
the $k$-th order approximations of vector fields of filtration degrees $k=0,-1,-2$. 

We claim that $\omega_{[2]}|_N$ is symplectic, and in fact coincides with the given symplectic structures on the vector bundles
$TN^\omega/TN$ and on $TN\oplus T^*N$ (the latter given by the pairing). This claim will prove the proposition, 
since it implies that $\omega_{[2]}$ nondegenerate on an open neighborhood of $N$, and hence everywhere, by homogeneity. 

For vector fields $X,X'$ of filtration degree $0$, vector fields $Y,Y'$ of filtration degree $-1$, and $Z,Z'$ of filtration degree $-2$ we have, for degree reasons,  
\[  \omega_{[2]}(Y_{[-1]},Z_{[-2]})|_N=0,\  \ \  \ \omega_{[2]}(X_{[0]},Y_{[-1]})|_N=0,\] 
\[\omega_{[2]}(X_{[0]},X'_{[0]})|_N=0,\ \ \ \ 
 \omega_{[2]}(Z_{[-2]},Z'_{[-2]})|_N=0.\]
 (Here we used that homogeneous functions of homogeneity degree $k<0$ are zero, and those of homogeneity $k>0$ restrict to $0$ on $N$.  
This shows that the first and third summand in 
 \eqref{eq:decomp}  are isotropic, and both are  is $\omega_{[2]}|_N$-orthogonal to middle  summand. 
On the other hand, 
\[ \omega_{[2]}(Y_{[-1]},Y'_{[-1]})|_N=\omega(Y,Y')|_N=\omega|_N(Y|_N,Y'|_N).\]
shows that $\omega_{[2]}$ restricts to the given symplectic structure on middle summand $TN^\omega/TN$. 
Finally, 
\[   \omega_{[2]}(X_{[0]},Z_{[-2]})|_N=\omega(X,Z)|_N=\omega|_N(X|_N,Z|_N).\]
But $X$ having filtration degree $0$ means in particular that it is tangent to $N$, and so 
the right hand side is just the pairing between a vector tangent to $N$ with a section of 
$TM|_N/TN^\omega$.
\end{proof} 
  
We think of $\nu_\W(M,N)$ with the symplectic form $\omega_{[2]}$ as the local model for the symplectic structure near the isotropic submanifold $N$.

\begin{theorem}[Isotropic embeddings] 
For any isotropic submanifold $N\subset M$ of a symplectic manifold $(M,\omega)$, with the corresponding 
weighted normal bundle $\nu_\W(M,N)$, there exists a tubular neighborhood embedding $\nu_\W(M,N)\supset O\stackrel{\varphi}{\lra} M$ satisfying 
\[ \varphi^*\omega=\omega_{[2]}.\] 
\end{theorem} 
 \begin{proof}

We follows the proof of Weinstein's  Lagrangian neighborhood theorem in Section \ref{subsec:weinstein}. 
By replacing $M$ with the image of an initial tubular neighborhood embedding $\nu_\W(M,N)\supset O_1\stackrel{\varphi_1}{\lra} M$, we may assume that $M$ is a star-shaped open neighborhood of $N$ inside $\nu_\W(M,N)$. Let $\alpha$ be the primitive of $\omega$, defined by the homotopy operator: 
\[  \alpha=\int_0^1 \f{1}{t}\ \kappa_t^*\iota_\E \omega\ \d t.\]
Since $\iota_\E$ and $\kappa_t^*$ preserve  filtration degrees, the 1-form $\alpha$ has filtration degree $2$, and 
the vector field  
 $X\in\mf{X}(M)$ be defined by 
\[ \iota_X\omega=2\alpha\]
$X$ has filtration degree $0$. The calculation 
\[
\iota_{X_{[0]}}\omega_{[2]}=2\alpha_{[2]}
=2 \int_0^1 \f{1}{t}\ \kappa_t^*\iota_\E \omega_{[2]}\ \d t
=2  \int_0^1 t\,\iota_\E \omega_{[2]}\ \d t=\iota_\E \omega_{[2]}
\]
shows that  $X_{[0]}=\E$, so that $X$ is (weighted) Euler-like for the triple $(M,N,TN^\omega/TN)$.
 Let 
$\nu_\W(M,N)\supset O\stackrel{\varphi}{\lra} M$ be the new tubular neighborhood embedding defined by $X$. 
 Then 
 \[ \L_\E\varphi^*\omega=\varphi^*\L_X\omega=\varphi^*\d\iota_X\omega=2\varphi^*\d\alpha=2\varphi^*\omega.\]
 Thus, $\varphi^*\omega$ is homogeneous of degree $2$, and hence coincides with its 2nd-order approximation 
 $\omega_{[2]}$. 	
 \end{proof}
 
 
\begin{remark}
In the standard treatment of Weinstein's isotropic embedding theorem \cite{wei:sym,we:nei}, one uses a connection to extend the given fiberwise symplectic structure on 
\[ T(\on{gr}(\nu(M,N))|_N=TN\oplus T^*N\oplus  TN^\omega/TN\] 
to a symplectic structure on the total space of $\on{gr}(\nu(M,N))$, and then establishes a symplectomorphism with a neighborhood of $N$ inside $M$.  The use of the weighted normal bundle removes the non-canonical choice of a connection.
\end{remark} 

\subsection{A symplectic interpretation of the bundle $\sA_2$}
Towards the end of Section \ref{subsec:weight}, we remarked that the weighted normal bundle $\nu_\W(M,N)$ 
for a given $F\subset \nu(M,N)$ is fully determined, as a graded bundle, by the degree $2$ component of the algebra 
bundle $\sA$, together with the inclusion map 
\[ m\colon \on{Sym}^2 F^*\to \sA_2.\] 
The vector bundle $\sA_2$ was defined in terms of its sheaf of sections, $\J/\I\J$. One may wonder about more geometric descriptions of $\sA_2$. As it turns out, there is an interesting symplectic interpretation. 
We begin with a symplectic interpretation of the second cotangent bundle. (See, e.g., \cite{we:le,bry:ovr}.) 

Given a symplectic vector space $(S,\omega)$, let $\on{Gr}_{\on{Lag}}(S)$ be its manifold of Lagrangian subspaces. 
For a fixed Lagrangian subspace $L\subset S$, the open subset $\on{Gr}_{\on{Lag}}(S,L)$ of Lagrangian subspaces that are transverse to $L$ is canonically an affine space, with $\on{Sym}^2 L$ as its space of motions.  This generalizes to symplectic vector bundles and Lagrangian subbundles, and in particular applies to the tangent bundles of symplectic manifolds with Lagrangian foliations. Consider in particular a cotangent bundle $T^*M$, with its standard symplectic form $\omega_{\on{can}}$, and let $V\subset T(T^*M)$ be the tangent bundle to the fibers of $T^*M\to M$. There is a canonical isomorphism of fiber bundles over $T^*M$, 
\begin{equation}\label{eq:secondcotangent} T^*_2M\stackrel{\cong}{\lra} \on{Gr}_{\on{Lag}}(T(T^*M),V),\end{equation}
where $T^*_2M$ is the second cotangent bundle. In more detail,  recall that $T^*_2M\to M$ is the vector bundle whose fiber fiber
at $p\in M$ is the space of 2-jets of functions $f\in C^\infty(M)$ with $f(p)=0$. The map taking 2-jets to 1-jets  gives a surjective vector bundle map $T^*_2M\to T^*M$, with kernel $\on{Sym}^2(T^*M)$. We will use this map to regard $T^*_2M$ as a fiber bundle over $T^*M$. On the other hand, every $f\in C^\infty(M)$ defines a Lagrangian submanifold of $T^*M$, given as the range of its exterior differential $\d f\colon M\to T^*M$. For $p\in M$, with $\xi=\d f|_p$,
the Lagrangian subspace 
\[ T_\xi(\on{ran}(\d f))\subset T_\xi(T^*M)\]
depends only on the 2-jet of 
$f$ at $p$. Adding a constant, we may arrange $f(p)=0$. In conclusion, the map \eqref{eq:secondcotangent} takes the 2-jet of a function $f$ with $f(p)=0$ to the Lagrangian subspace given as the tangent space to $\on{ran}(\d f)$ at
$\xi=\d f|_p$. The subbundle $\on{Sym}^2(T^*M)$ is realized as the restriction of $\on{Gr}_{\on{Lag}}(T(T^*M),V)$ to 
$M\subset T^*M$, using the splitting $T(T^*M)|_M=V|_M\oplus TM$ and the symplectic form to identify $V|_M\cong T^*M$.

Now suppose that $N\subset M$ is a closed submanifold, and that $F\subset \nu(M,N)$ is a subbundle of its normal bundle.
The conormal bundle $\nu(M,N)^*$ is a Lagrangian submanifold of $T^*M$, and hence $\on{ann}(F)\subset \nu(M,N)^*$ is an isotropic submanifold of $T^*M$. Let $C=T\on{ann}(F)^\omega$, a coisotropic subbundle along $\on{ann}(F)$, and take $S$ to be the symplectic vector bundle  
\begin{equation}\label{eq:symplnormal} C/C^\omega\to \on{ann}(F).\end{equation}
It comes with a distinguished Lagrangian subbundle $L\subset C/C^\omega$
obtained by symplectic reduction of the vertical bundle $V\subset T(T^*M)$:
\[ L= (V|_{\on{ann}(F)}\cap C)/(V|_{\on{ann}(F)}\cap C^\omega).\] 
We obtain an identification 
\[ \sA_2\stackrel{\cong}{\lra} \on{Gr}_{\on{Lag}}(C/C^\omega,L).\]
The map is constructed similar to \eqref{eq:secondcotangent}. Let $f\in C^\infty(M)_{(2)}$, so that $f|_N=0$ and 
$\d f|_{\ti{F}}=0$. For $p\in N$, the differential $\xi=\d f|_p$ is then an element of $\on{ann}(F)$. The 
tangent space to $\on{ran}(\d f)\subset T^*M$ at $\xi$ is a Lagrangian subspace of $T_\xi(T^*M)$ transverse to $V|_\xi$; 
by reduction it defines a Lagrangian subspace of $(C/C^\omega)|_\xi$ transverse to $L|_\xi$. By checking in adapted coordinates (as in Remark \ref{rem:feeling}), one verifies that this gives an isomorphism between $\sA_2|_\xi$ with 
$\on{Gr}_{\on{Lag}}((C/C^\omega)|_\xi,L|_\xi)$. The subspace $\on{Sym}^2 F^*\to \sA_2$ is realized as the restriction of the bundle 
$\on{Gr}_{\on{Lag}}(C/C^\omega,L)$ to the submanifold $N\subset\on{ann}(F)$.

 \bibliographystyle{amsplain} 
\def\cprime{$'$} \def\polhk#1{\setbox0=\hbox{#1}{\ooalign{\hidewidth
			\lower1.5ex\hbox{`}\hidewidth\crcr\unhbox0}}} \def\cprime{$'$}
\def\cprime{$'$} \def\cprime{$'$} \def\cprime{$'$} \def\cprime{$'$}
\def\polhk#1{\setbox0=\hbox{#1}{\ooalign{\hidewidth
			\lower1.5ex\hbox{`}\hidewidth\crcr\unhbox0}}} \def\cprime{$'$}
\def\cprime{$'$} \def\cprime{$'$} \def\cprime{$'$} \def\cprime{$'$}
\providecommand{\bysame}{\leavevmode\hbox to3em{\hrulefill}\thinspace}
\providecommand{\MR}{\relax\ifhmode\unskip\space\fi MR }
\providecommand{\MRhref}[2]{%
	\href{http://www.ams.org/mathscinet-getitem?mr=#1}{#2}
}
\providecommand{\href}[2]{#2}

\end{document}